\documentclass[english]{amsart}
\usepackage{graphicx} 
\usepackage{amsmath,amssymb,amsthm,xcolor}
\usepackage{graphicx,tikz}
\usepackage[colorlinks=true, allcolors=blue]{hyperref}
\usepackage{circuitikz}
\usepackage{listings}
\usepackage{xcolor}
\usepackage{multirow}
\usepackage{float}

\usepackage{geometry}
\definecolor{codegreen}{rgb}{0,0.6,0}
\definecolor{codegray}{rgb}{0.5,0.5,0.5}
\definecolor{codepurple}{rgb}{0.58,0,0.82}
\definecolor{backcolour}{rgb}{0.95,0.95,0.92}

\lstdefinestyle{mystyle}{
    backgroundcolor=\color{backcolour},   
    commentstyle=\color{codegreen},
    keywordstyle=\color{magenta},
    numberstyle=\tiny\color{codegray},
    stringstyle=\color{codepurple},
    basicstyle=\ttfamily\footnotesize,
    breakatwhitespace=false,         
    breaklines=true,                 
    captionpos=b,                    
    keepspaces=true,                 
    numbers=left,                    
    numbersep=5pt,                  
    showspaces=false,                
    showstringspaces=false,
    showtabs=false,                  
    tabsize=2
}

\usetikzlibrary{shapes.geometric}

\renewcommand{\epsilon}{\varepsilon}

\usepackage{cjhebrew}

\definecolor{scarlet}{rgb}{1.0, 0.13, 0.0}
\definecolor{cardinal}{HTML}{840A2D}
\definecolor{illiniorange}{HTML}{F38025}

\theoremstyle{theorem}
\newtheorem{thm}{Theorem}[section]

\newtheorem{lemma}[thm]{Lemma}

\theoremstyle{definition}
\newtheorem{defn}[thm]{Definition}
\newtheorem{example}[thm]{Example}
\newtheorem{question}[thm]{Question}
\newtheorem{remark}[thm]{Remark}

\title{The Edge-Distinguishing Game}
\author[Benjamin, Benthem, Burkel, Chesser, Janssen]{Nathaniel Benjamin, Elisa Benthem, Cooper Burkel, Marissa Chesser, and Mike Janssen}
\date{\today}

\begin{document}

\maketitle

\begin{abstract}
In this paper, we introduce a graph coloring game called the Edge-Distinguishing Game (EDGe). The edge-distinguishing chromatic number of a graph is used to determine the moves each player can make. We determine which player has a winning strategy for particular graphs and graph families. Additionally, utilizing principles from game theory as well as previous work \cite{garcia2025computational} on a computational solution for the Game of Cycles. 
    
\end{abstract}

\section{Introduction}\label{sec:intro}

The concept of an edge-distinguishing coloring of a graph was first introduced by Frank et al. in 1982 \cite{frank1982edcn}. Subsequent authors used this method of vertex coloring to determine the \emph{edge-distinguishing chromatic number} (EDCN) for various families of graphs \cite{al1988edge, fickes2021edge} such as paths, cycles, petal graphs, and spider graphs with 4-legs.

In this context, a coloring of a graph is defined in the following way. Let $[k] = \{1,2,3,\dots,k\}$ be the set of colors, and let $\big(\binom{[k]}{i}\big)$ denote the $i$-subsets of $[k]$, where each set is a multiset.

\begin{defn}\label{def:k-coloring}
    Let $G$ be a graph with vertex set $V(G)$ and edge set $E(G)$. A \emph{$k$-coloring} of a graph, is a labeling $c: V(G) \rightarrow [k]$. The \emph{induced edge-coloring} $c': E(G) \rightarrow$ $\big(\binom{[k]}{2}\big)$ is defined by $c'(\{v_i, v_j\}) = \{c(v_i), c(v_j)\}$. A \emph{partial $k$-coloring} is a $k$-coloring on a subset of the vertices of $G$, i.e. $c:U\to[k]$ for $U\subseteq V(G)$. A \emph{partial induced edge-coloring} is the induced edge coloring of a partial $k$-coloring on the induced subgraph $G[U]$.
\end{defn}
An edge-distinguishing coloring can then be defined as follows. 

\begin{defn}\label{def:edge-distinguishing}
    Let $c$ be a $k$-coloring of a graph $G$, and $c'$ the induced edge-coloring. If $c'$ is injective, $c$ is called \emph{edge-distinguishing}. The \emph{edge-distinguishing chromatic number} (EDCN) of $G$, denoted $\lambda(G)$, is the smallest integer $k$ such that an edge-distinguishing $k$-coloring exists. 
\end{defn}

An example of an edge-distinguishing coloring of a graph is shown in Figure \ref{fig:P_5propergraph}. 

    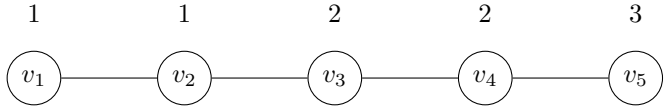
\begin{figure}[H]
    \centering
    
    \begin{tikzpicture}
    \node[draw,circle] (v1) {$v_1$};
    \node at ([shift={(90:0.5)}]v1.90) {1}; 
    
    \path (v1) ++(0:2) node [draw,circle] (v2) {$v_2$};
    \node at ([shift={(90:0.5)}]v2.90) {1}; 

    \path (v1) ++(0:4) node [draw,circle] (v3) {$v_3$};
    \node at ([shift={(90:0.5)}]v3.90) {2}; 

    \path (v1) ++(0:6) node [draw,circle] (v4) {$v_4$};
    \node at ([shift={(90:0.5)}]v4.90) {2}; 

    \path (v1) ++(0:8) node [draw,circle] (v5) {$v_5$};
    \node at ([shift={(90:0.5)}]v5.90) {3}; 

    \draw (v1) -- (v2) -- (v3) -- (v4) -- (v5);
    \end{tikzpicture}
    \caption{$P_5$ with an edge-distinguishing $3$-coloring.}\label{fig:P_5propergraph}
    \end{figure}

The primary purpose of this work is to introduce and investigate a game that utilizes the edge-distinguishing concept underlying the EDCN. 
We define the game as follows.

\begin{defn}\label{def:EDGe-rules}
    Let $G$ be a finite simple graph, sometimes called the \emph{game board}, with $\lambda(G) = k$.
    The \emph{\underline{e}dge-\underline{d}istinguishing \underline{g}am\underline{e}}, or \emph{EDGe}, is a two-player game played on $G$ according to the following rules.
    \begin{enumerate}
        \item Players take turns labeling uncolored vertices with an available color from the list $[k]$.
        \item A \emph{legal move} is one where the color played results in an edge-distinguishing partial induced edge coloring on the set of colored vertices. 
        \item The player that makes the last legal move wins.
    \end{enumerate}
    If, during the course of a game, a given vertex can no longer be colored without inducing a nonunique edge coloring, we call that vertex \emph{unmarkable}.
\end{defn}

\begin{example}\label{ex:EDGe-gameplay}
    According to \cite{frank1982edcn, al1988edge}, $\lambda(P_5) = 3$.
    Consider the game state in Figure \ref{fig:P_5partial} that could occur partway through a game of EDGe on $P_5$.

    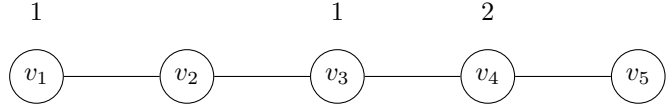
\begin{figure}[H]
        \centering
        \begin{tikzpicture}
        \node[draw,circle] (v1) {$v_1$};
        \node at ([shift={(90:0.5)}]v1.90) {1}; 
        
        \path (v1) ++(0:2) node [draw,circle] (v2) {$v_2$};
    
        \path (v1) ++(0:4) node [draw,circle] (v3) {$v_3$};
        \node at ([shift={(90:0.5)}]v3.90) {1}; 
    
        \path (v1) ++(0:6) node [draw,circle] (v4) {$v_4$};
        \node at ([shift={(90:0.5)}]v4.90) {2}; 
    
        \path (v1) ++(0:8) node [draw,circle] (v5) {$v_5$};
    
        \draw (v1) -- (v2) -- (v3) -- (v4) -- (v5);
        \end{tikzpicture}
        \caption{A partially played $P_5$.}\label{fig:P_5partial}
    \end{figure}

Because there are three vertices labeled, it is Player 2's turn. We see that $v_2$ is an unmarkable vertex because it is adjacent to two vertices of the same color.  
Therefore, the only \emph{markable vertex} is $v_5$. We know from the game board that vertices $v_3$ and $v_4$ induce an edge coloring of $\{1, 2\}$, so Player 2 cannot label $v_5$ with $1$ as it will induce that same edge color. Instead, Player 2 can label $v_5$ with $3$ and thus win the game, since no other vertex can be legally labeled. As shown here, we will usually assume without loss of generality that if a player plays a new color that has not yet been played, then the first unused color is played.
\end{example}

We seek to determine which player has a \emph{winning strategy} for various graphs and families of graphs.
Since EDGe is a finite, impartial, 2-player game with no draws, Zermelo's Theorem \cite{Amir2017} applies, implying that one player must have a winning strategy.

In Section \ref{sec:computation}, we briefly discuss code that produces the winning player on any graph. In Section \ref{sec:results}, we discuss some unique strategies and solutions for certain graphs and families; our results from this section are summarized in Table \ref{tab:results}.

\begin{table}[h!]\label{tab:results}
\centering
    \begin{tabular}{ |c|c|c|  }
        \hline
        Graphs& Player 1& Player 2\\
        \hline
        Complete & $K^*_{2n+1}$ & $K_n, K^*_{2n}$\\
        Complete Bipartite & $K_{n,m}$ ($m+n$ odd) & $K_{n,m}$ ($m+n$ even)\\
        Wheel & $W_{2n+1}$ & $W_{2n}$\\
        Paths & $P_1, P_3, P_4, P_5$ & $P_2, P_6, P_7^\dagger$\\
        Cycles & $C_5, C_6$ & $C_3, C_4, C_7, C_8^\dagger$\\
        Chorded Cycles$^\dagger$ & $C_5^{\{1,3\}}, C_6^{\{1,3\}}, C_8^{\{1,3\}}$ & $C_4^{\{1,3\}}, C_7^{\{1,3\}}$\\
        Triangular Ladder & $T_1$ & $T_2, T_3, T_4, T_5^\dagger, T_6^\dagger, T_7^\dagger, T_8^\dagger$\\
        Miscellaneous & Moser Spindle, Petersen$^\dagger$, Envelope$^\dagger$ & Cube$^\dagger$, Octahedron$^\dagger$, Book\\
        \hline
    \end{tabular}
    \caption{Summary of Results. $\dagger$ represents graphs whose winnability was checked by computer.}
    \label{table:summary}
\end{table}

In addition to determining which players have winning strategies for various graphs, we also show in Section \ref{sec:results} that the requirement in Definition \ref{def:EDGe-rules} that $\lambda(G)$ colors should be used when playing the game on $G$ is not extraneous as it is possible that the player with the winning strategy on $G$ can change if more colors are allowed.

\begin{thm}[Theorems \ref{thm:P6} and \ref{thm:numcolors}]
    On $P_6$, Player 1 has the winning strategy with $\lambda(P_6)$ colors, while Player 2 has the winning strategy if $\lambda(P_6)+1$ colors are used. In particular, for a fixed game board, the number of colors available can change which player has the winning strategy.
\end{thm}

We conclude the paper with ideas for future work in Section \ref{sec:future-work}.

\medskip

\noindent\textbf{Acknowledgements.} This work was completed while the second and third authors were undergraduate research students during the summer of 2024. All authors wish to gratefully thank the Dordt University Kielstra Center for Research and Grants for its generous financial support for the completion of this project.

\section{A Game-Theoretic Perspective}\label{sec:computation}

We adapted a Python program used to determine winning strategies for the Game of Cycles \cite{garcia2025computational} to apply to EDGe. The program returns which player has the winning strategy for a given graph. In addition, the code produces a game tree, which is a directed graph that describes all of the legal game states as well as the possible moves from state to state. An example for the game on $P_3$ is provided in Figure \ref{fig:P_3digraph}. 

\begin{figure}[h]
    \centering
    \includegraphics[scale=0.65]{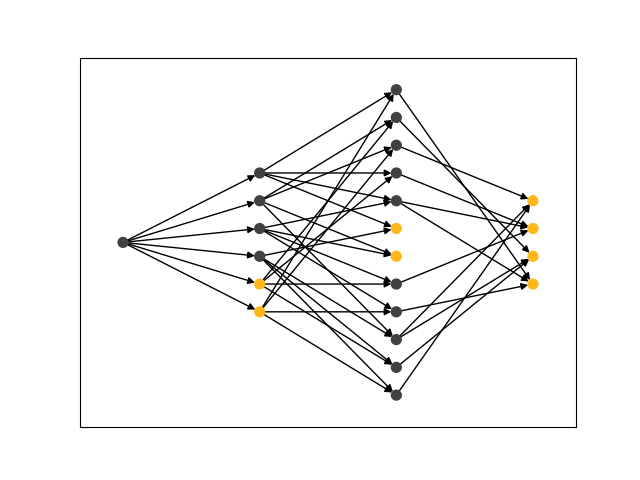}
    \caption{The game tree digraph for $P_3$}
    \label{fig:P_3digraph}
\end{figure}

Each node in the digraph represents a possible game state, with the initial node representing the empty board. Each edge indicates that it is legal for a player to move from the games state of the tail node to the game state of the head node. We can then label each node based on which player is in a winning position.

\begin{defn}\label{def:P-N}
    A $P$-\emph{position} means the \emph{previous} player, the one who has just played, is in a winning position. An $N$-\emph{position} means the \emph{next} player to play is in a winning position. 
\end{defn}

For an impartial game, every game state must be either a $P$-position or an $N$-position.

\begin{defn}\label{def:terminating-playable}
    \emph{Terminating boards} must either be fully colored or every uncolored vertex must be unmarkable. In any other case, the board is considered \emph{playable} or \emph{markable}. 
\end{defn}

Boards in a $P$-position must terminate or lead to $N$-position boards, while boards in an $N$-position lead to at least one $P$-position board.

This means that the goal of both players is to put the board into a $P$-position after the completion of their move. Since $P$-positions only lead to $N$-positions, once a player has made that move, it is impossible for the next player to change who is in a winning position. The only way to alter who wins the game is for a player to make a move to an $N$-position when a $P$-position is available, because it allows for the other player to put the game in their favor. In Figure \ref{fig:P_3digraph}, we have colored all $P$-positions as \emph{yellow} and all $N$-positions \emph{gray}. As the initial node represents the empty board, if the initial node is gray (i.e. an $N$-position), Player 1 has a winning strategy, while a yellow initial node (i.e. a $P$-position), indicates Player 2 has the winning strategy. Therefore, whichever player has the winning strategy can use the directed graph to follow the moves which lead to game states marked by yellow nodes, ensuring victory.

Our code can be found at \cite{edge_game_python} and the main parts of the algorithm can be found in the file \texttt{EDGe-Algorithm-Supplement.pdf}, posted in that GitHub repository. We encourage the reader to explore the game on their own.

\section{Results}\label{sec:results}

In this section, we determine which player has the winning strategy for several entire families of graphs, some particular cases within larger families, and for a number of other specific graphs.

\begin{remark}
    Boards that are isomorphic play the same in EDGe. This fact will be used without comment in several of the proofs in this section.
\end{remark}

\subsection{Completely Solved Families}\label{subsec:solved-families}

In this first subsection, we determine which player has the winning strategy for complete graphs, complete bipartite graphs, wheel graphs, and triangular book graphs.

\begin{thm}\label{thm:complete}
    For the complete graph $K_n$ with $n \ge 2$, Player 2 has the winning strategy. 
\end{thm}

\begin{proof}
    When $n=2$, we know from \cite{frank1982edcn} that $\lambda(K_2) = 1$ . Therefore, Player 2 simply plays a 1 on the vertex Player 1 did not color, completely coloring the vertices and winning the game.
    
    For $n \ge 3$, it is shown in \cite{frank1982edcn} that $\lambda(K_n) = n$. No matter which color Player 1 plays, if Player 2 plays the same color on another vertex, every other vertex becomes unmarkable. 
    This is because every other vertex is adjacent to these two vertices, so any color Player 1 could attempt to play as their second move will be a part of a triangle where two edges have the same color. Therefore, Player 2 wins by being the last player to be able to make a legal move.
\end{proof}

\begin{thm}\label{thm:bipartite}
    For complete bipartite graphs, $K_{n, m}$, Player 1 has the winning strategy if $n+m$ is odd, and Player 2 has the winning strategy if $n+m$ is even.
\end{thm}

\begin{proof}
    Assume that we have a complete bipartite graph $K_{n, m}$ with vertex bipartition $V = V_1\cup V_2$. We know from \cite{frank1982edcn} that $\lambda(K_{n, m}) = (n+m) - 1$.
    \item\underline{Case 1:} \emph{The total number of vertices is odd.} If the total number of vertices is odd, this means that the size of exactly one of the parts is odd, say $V_1$ without loss of generality. Player 1's winning strategy is to play $1$ on a vertex from $V_1$. If Player 2 plays on a vertex in $V_1$, Player 1 need only play 1 on another vertex in $V_1$, rendering all vertices in the other side unmarkable, since there are two edges coming from the same colored vertices to each vertex in $V_2$. This guarantees that there will be an even number of markable vertices left to begin Player 2's turn. No further vertices are able to be eliminated by being made unmarkable, since no edges can be assigned a color, thus Player 1 will win by being the last to play. If instead Player 2 plays on a vertex in $V_2$, then Player 1 can force the graph to be completed with a proper edge-distinguishing $k$-coloring. This is because either Player 2 colored a vertex with 1, or Player 1 can color a vertex in $V_2$ with 1, since there are at least two vertices in $V_2$. Notice then that once a color $i$ is used it creates a $\{1,i\}$-edge because there is a vertex colored 1 in $V_1$ and in $V_2$. Hence, each color can be used only once, and there are enough colors remaining to do so, meaning no vertices can be made unmarkable. This is advantageous to Player 1 because there are an odd number of vertices remaining, giving this player the last move. 
    \item\underline{Case 2:} \emph{The total number of vertices is even.} If the total number of vertices is even, Player 2 will win by playing the same color as Player 1 but on a vertex in the opposite set of vertices. Doing so will force the rest of the board to be played to create a proper edge-distinguishing $k$-coloring by the same argument as before. Since Player 1's turn begins with an even number of vertices left, Player 2 will win by completing the board.
\end{proof}

Next, we investigate wheel graphs, which is the join of a cycle with a single vertex, called the \emph{hub}, i.e. $W_n := C_{n-1}\nabla K_1$. First, we need the EDCN.

\begin{lemma}
    $\lambda(W_n) = n$, for all $n$.
\end{lemma}

\begin{proof}
    Each spoke vertex in a wheel graph is adjacent to the hub vertex, so no two spoke vertices can be the same color. If a spoke shares the same color as the hub, each other vertex adjacent to that spoke can no longer be colored, since they are adjacent to two vertices of the same color. Therefore, there needs to be a unique color for every vertex, meaning $\lambda(W_n) = n$.
\end{proof}

\begin{thm}\label{thm:wheel}
    For the wheel graph on $n$ vertices, $W_n$, for $n \ge 4$, Player 1 has the winning strategy if $n$ is odd, whereas Player 2 has the winning strategy if $n$ is even.
\end{thm}

\begin{proof}
    By the previous lemma, $\lambda(W_n) = n$. 
    \item\underline{Case 1:} \emph{$n$ is odd.} If $n$ is odd, Player 1 has the winning strategy by playing $1$ in the middle vertex. If Player 2 then plays $1$ on one of the outside vertices, this move makes both vertices adjacent to Player 2's move unmarkable. No other vertices can be made unmarkable because all remaining vertices are adjacent to the $1$ in the middle, so each color can be played only once. Since there are an odd number of vertices remaining,  Player 1 wins. 
    Alternatively, if Player 2 plays a different color, say $2$, on an outside vertex, Player 1 wins by playing $1$ on a non-adjacent outside vertex, of which there is at least one, since $n\ge 5$. This move again renders the two vertices adjacent to the spoke vertex colored 1 unmarkable,and leaves an even number of markable vertices on Player 2's turn. As before, no other vertices can be made unmarkable from this configuration. Therefore, Player 1 wins if $n$ is odd.
    \item\underline{Case 2:} \emph{$n$ is even.} If $n$ is even, Player 2 has the winning strategy by playing the same color as Player 1, either on the hub vertex if Player 1 played on a spoke, or on a spoke vertex if Player 1 played on the hub. This once again makes the two vertices adjacent to the two colored vertices unmarkable. This leaves only an even number of vertices, none of which can be made unmarkable since all the remaining vertices are adjacent to the $1$ on the hub. Since Player 1 is left with an even number of vertices, Player 2 is in the winning position. Therefore, if $n$ is even, Player 2 has the winning strategy.
\end{proof}

Now consider triangular book graphs, defined as the join of an edge and the empty graph, i.e. $B_n := K_2\nabla \overline{K_{n-2}}$.

\begin{lemma}\label{thm:book-edcn}
    For a (triangular) book graph of $n$ vertices, $B_n$, $\lambda(B_n)=n$.
\end{lemma}

\begin{proof}
    Because all $K_3$ subgraphs of a triangular book graph $B_n$ share two common vertices and the EDCN of $K_3$ is $3$, it follows that a unique color is required for every vertex of $B_n$. Therefore, $\lambda(B_n) = n$.
\end{proof}

\begin{thm}\label{thm:book}
    For the (triangular) book graph on $n$ vertices, $B_n$, for $n\ge 3$, Player 2 has the winning strategy.
\end{thm}

\begin{proof} 
    Assume without loss of generality that Player 1's first move is a $1$. If Player 1 plays on the spine, Player 2 will win by playing $1$ on the other vertex of the spine. Doing so makes all other vertices unmarkable, thus causing Player 2 to win the game. If Player 1 does not play on the spine, Player 2's winning strategy has two cases depending on if $n$ is even or odd.
    \item\underline{Case 1:} \emph{$n$ is odd.} If $n$ is odd, Player 2 will win by playing the same color, $1$, on one of the vertices on the spine (see Figure \ref{fig:BookgraphCase1}). By doing so, this makes the other spine vertex unmarkable. Since three vertices have been played on or eliminated in two moves, this means Player 1 will have an even number of markable vertices remaining. None of the remaining vertices can be made unmarkable, since they are all adjacent to a vertex colored 1, so each color can only be used once. Player 2 will win by being the last one to play.
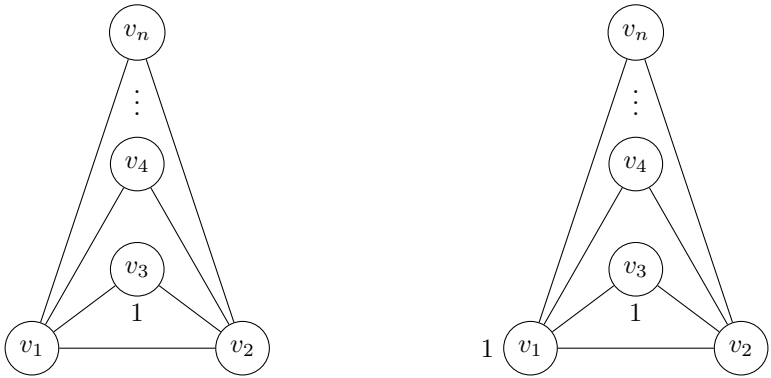
\begin{figure}[H]
\centering
    \begin{minipage}{0.4\textwidth}
    \begin{tikzpicture}[scale=0.7]
    \node[draw,circle] (v1) {$v_1$};
    
    \path (v1) ++(0:4) node [draw,circle] (v2) {$v_2$};

    \path (v1) ++(0:2) node [] (u) {};

    \path (u) ++(90:1.5) node [draw,circle] (v3) {$v_3$};
    \node at ([shift={(270:0.3)}]v3.270) {1}; 

    \path (u) ++(90:3.5) node [draw,circle] (v4) {$v_4$};
    \node at ([shift={(90:0.8)}]v4.90) {$\vdots$}; 

    \path (u) ++(90:6) node [draw,circle] (v5) {$v_n$};

    \draw (v1) -- (v2) -- (v3) -- (v1);
    \draw (v1) -- (v4) -- (v2);
    \draw (v1) -- (v5) -- (v2);
    \end{tikzpicture}
    \end{minipage}
    \begin{minipage}{0.3\textwidth}
    \begin{tikzpicture}[scale=0.7]
    \node[draw,circle] (v1) {$v_1$};
    \node at ([shift={(180:0.3)}]v1.180) {1}; 
    
    \path (v1) ++(0:4) node [draw,circle] (v2) {$v_2$};

    \path (v1) ++(0:2) node [] (u) {};

    \path (u) ++(90:1.5) node [draw,circle] (v3) {$v_3$};
    \node at ([shift={(270:0.3)}]v3.270) {1}; 

    \path (u) ++(90:3.5) node [draw,circle] (v4) {$v_4$};
    \node at ([shift={(90:0.8)}]v4.90) {$\vdots$}; 

    \path (u) ++(90:6) node [draw,circle] (v5) {$v_n$};

    \draw (v1) -- (v2) -- (v3) -- (v1);
    \draw (v1) -- (v4) -- (v2);
    \draw (v1) -- (v5) -- (v2);
    \end{tikzpicture}
    \end{minipage}

\vspace{0.5cm}
\caption{$B_n$, Case 1, where Player 2 is in a winning position.}
\label{fig:BookgraphCase1}
\end{figure}
    \item\underline{Case 2:} \emph{$n$ is even.} If $n$ is even, Player 2 has the winning strategy by playing the same color, $1$ on another vertex not on the spine (see Figure \ref{fig:BookgraphCase2}). By doing so, the two vertices on the spine are made unmarkable. None of the remaining vertices can be made unmarkable because there are no edges between them, so no edge can be assigned a color. Since there are an even number of vertices left to start Player 1's turn, Player 2 will win by being the last to play.
\end{proof}
\begin{figure}[H]
\centering
\begin{minipage}{0.4\textwidth}
    \begin{tikzpicture}[scale=0.7]
    \node[draw,circle] (v1) {$v_1$};
    
    \path (v1) ++(0:4) node [draw,circle] (v2) {$v_2$};

    \path (v1) ++(0:2) node [] (u) {};
    \node at ([shift={(90:0.5)}]u.90) {1}; 

    \path (u) ++(90:1.5) node [draw,circle] (v3) {$v_3$};
    \node at ([shift={(90:0.8)}]v3.90) {$\vdots$}; 

    \path (u) ++(90:4) node [draw,circle] (v5) {$v_n$};

    \draw (v1) -- (v2) -- (v3) -- (v1);
    \draw (v1) -- (v5) -- (v2);
    \end{tikzpicture}
    \end{minipage}
    \begin{minipage}{0.3\textwidth}
    \begin{tikzpicture}[scale=0.7]
    \node[draw,circle] (v1) {$v_1$};
    
    \path (v1) ++(0:4) node [draw,circle] (v2) {$v_2$};

    \path (v1) ++(0:2) node [] (u) {};

    \path (u) ++(90:1.5) node [draw,circle] (v3) {$v_3$};
    \node at ([shift={(270:0.3)}]v3.270) {1}; 
    \node at ([shift={(90:0.8)}]v3.90) {$\vdots$}; 

    \path (u) ++(90:4) node [draw,circle] (v5) {$v_n$};
    \node at ([shift={(90:0.5)}]v5.90) {1}; 

    \draw (v1) -- (v2) -- (v3) -- (v1);
    \draw (v1) -- (v5) -- (v2);
    \end{tikzpicture}
    \end{minipage}
    \vspace{0.5cm}
    
\caption{$B_n$, Case 2, where Player 2 is in a winning position.}
\label{fig:BookgraphCase2}
\end{figure}
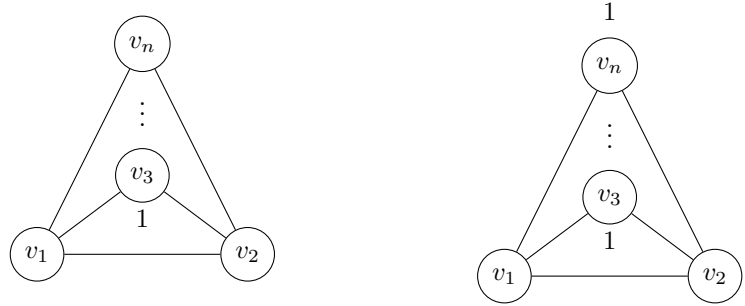

\subsection{Paths}\label{subsec:paths}

\noindent In this section, we investigate winning strategies on paths. We show that Player 1 has the winning strategy for $P_n$ for $3\le n\le 6$. Note that we assume in this section that the vertices on each path $P_n$ are labeled from left to right as indicated in Figure \ref{fig:path-label}.

\begin{figure}[H]
    \centering

    \begin{tikzpicture}[scale=0.8,every node/.style={draw=black,circle}]
    \node (v1) at (0,0) {$v_1$};
    \node (v2) at (2,0) {$v_2$};
    \node[draw=none] (dots) at (4,0) {$\cdots$};
    \node (vn) at (6,0) {$v_n$};
    \draw[-] (v1) to (v2) to (dots) to (vn);
    \end{tikzpicture}

    \caption{Vertex labeling for paths $P_n$.}
    \label{fig:path-label}
\end{figure}
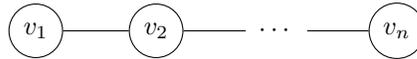

\begin{thm}\label{thm:P3}
    For $P_3$, Player 1 has the winning strategy.
\end{thm}

\begin{proof}
     Players have the ability to play two colors on this graph, as $\lambda(P_3) = 2$ \cite{frank1982edcn}. Player 1 plays $1$ on $v_2$. Player 2 can play $1$ or $2$ on either of the two remaining vertices. Then Player 1 plays the other color on the other vertex to win.
\end{proof}

\begin{thm}\label{thm:P4}
    Player 1 has the winning strategy on $P_4$.
\end{thm}

\begin{proof}
     Players have the ability to play two colors on this graph, as $\lambda(P_4) = 2$ \cite{frank1982edcn}. Player 1 begins by playing $1$ on $v_1$. If Player 2 plays on $v_2$ or $v_4$, then Player 1 plays the same color as Player 2 on the other one, making $v_3$ unmarkable and winning the game. If Player 2 plays a 1 or a 2 on $v_3$, then Player 1 colors $v_4$, leaving $v_2$ unmarkable, either for having two adjacent 1 vertices in the former case, or for having no viable color options in the latter.
     
\begin{figure}[H]
\centering

\begin{minipage}{0.5\textwidth}
    \begin{tikzpicture}[scale=0.7]
    \node[draw,circle] (v1) {$v_1$};
    \node at ([shift={(90:0.4)}]v1.90) {1}; 
    
    \path (v1) ++(0:2) node [draw,circle] (v2) {$v_2$};
    \node at ([shift={(90:0.4)}]v2.90) {$i$}; 

    \path (v2) ++(0:2) node [draw,circle] (v3) {$v_3$};
    
    \path (v3) ++(0:2) node [draw,circle] (v4) {$v_4$};
    \node at ([shift={(90:0.4)}]v4.90) {$i$}; 

    \draw (v1) -- (v2) -- (v3) -- (v4);
    \end{tikzpicture}
    \end{minipage}
    \begin{minipage}{0.3\textwidth}
    \begin{tikzpicture}[scale=0.7]
    \node[draw,circle] (v1) {$v_1$};
    \node at ([shift={(90:0.4)}]v1.90) {1}; 
    
    \path (v1) ++(0:2) node [draw,circle] (v2) {$v_2$};

    \path (v2) ++(0:2) node [draw,circle] (v3) {$v_3$};
    \node at ([shift={(90:0.4)}]v3.90) {$i$};
    
    \path (v3) ++(0:2) node [draw,circle] (v4) {$v_4$};
    \node at ([shift={(90:0.4)}]v4.90) {1}; 

    \draw (v1) -- (v2) -- (v3) -- (v4);
    \end{tikzpicture}
    \end{minipage}
    \vspace{0.5cm}
    
\caption{The two different end states for $P_4$, where $i\in\{1,2\}$.}
\label{fig:P_4game}
\end{figure}
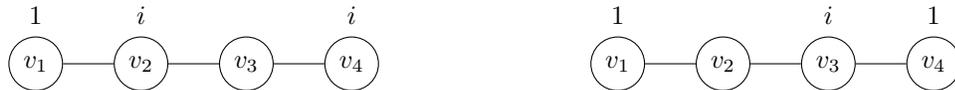
\end{proof}

\begin{thm}\label{thm:P5}
    Player 1 has the winning strategy on $P_5$.
\end{thm}

\begin{proof}
     Players have the ability to play 3 colors on this graph, as $\lambda(P_5) = 3$ \cite{frank1982edcn}. Player 1's first move is to play $1$ on $v_3$. From there, if Player 2 plays on $v_i$, Player 1's optimal move would be to play on $v_{6-i}$. This leads to 2 unique cases.
     \item \underline{Case 1:} \emph{Player 2 plays on $v_1$ or $v_5$.} The strategy described above indicates that if Player 2 plays on $v_5$, Player 1 should play on $v_1$ and vice versa. In this case Player 1 should always play the same color that Player 2 played. In the case where Player 2 played a 1, by playing another 1, Player 1 makes every other vertex unmarkable and thus wins the game. If Player 2 played another color, say 2, by playing their next move following the strategy Player 1 creates 2 vertices adjacent to both a 1 and a 2. This means that 1, 2, and 3 can be played on either vertex. Therefore, in order to win, Player 1 plays $3$ if Player 2 did not, and $1$ if Player 2 did play a $3$. This will complete the board and thus Player 1 wins.
     \item \underline{Case 2:} \emph{Player 2 plays on $v_2$ or $v_4$.} The strategy above indicates that if Player 2 plays on $v_2$, then Player 1 should play on $v_4$ and vice versa. Player 1 should play a 1 if Player 2 did not, and a 2 if Player 2 did, resulting in a configuration isomorphic to Figure \ref{fig:P_5games1}. By doing so, this will remove the ability for Player 2 to make any vertex unmarkable. Therefore, Player 1 will win by being the last to play by coloring $v_1$ or $v_5$ with 3, whichever is left.
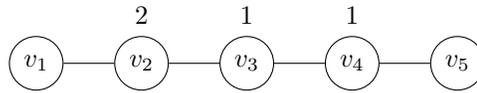
\begin{figure}[H]
\centering

\begin{tikzpicture}[scale=0.7]
    \node[draw,circle] (v1) {$v_1$};
    
    \path (v1) ++(0:2) node [draw,circle] (v2) {$v_2$};
    \node at ([shift={(90:0.4)}]v2.90) {2}; 

    \path (v2) ++(0:2) node [draw,circle] (v3) {$v_3$};
    \node at ([shift={(90:0.4)}]v3.90) {1};
    
    \path (v3) ++(0:2) node [draw,circle] (v4) {$v_4$};
    \node at ([shift={(90:0.4)}]v4.90) {1}; 
    
    \path (v4) ++(0:2) node [draw,circle] (v5) {$v_5$};

    \draw (v1) -- (v2) -- (v3) -- (v4) -- (v5);
    \end{tikzpicture}
\caption{$P_5$, Case 2, where Player 1 is in a winning position.}
\label{fig:P_5games1}
\end{figure}
\end{proof}

\begin{thm}\label{thm:P6}
    Player 1 has the winning strategy on $P_6$.
\end{thm}

\begin{proof}
     Players choose from 3 colors to play, as $\lambda(P_6) = 3$ \cite{frank1982edcn}. Player 1 plays $1$ on vertex $v_2$. The strategy for Player 1 for subsequent moves has two cases. 
     \item \underline{Case 1:} \emph{Player 2 plays somewhere other than $v_5$.} In this case, Player 1 is in the winning position by playing $1$ on $v_5$. Now every other vertex will be adjacent to a $1$. There are only three possible edge colorings containing $1$, namely $\{1, 1\}, \{1, 2\}$, and $\{1, 3\}$. The first move by Player 2 creates one of those edges, so Player 2's second move must use a different color, leaving one unused color. 
     Thus, Player 1 will win on the third move, with one unmarkable vertex remaining.
     
     \item \underline{Case 2:} \emph{Player 2 plays on $v_5$.} Note that if Player 2 plays a 1 on $v_5$, after Player 1's next turn, the game plays out identically to Case 1. If Player 2 does not play a 1, Player 1's strategy would then be to create an unmarkable vertex at $v_3$ by playing $1$ on $v_4$. This move will leave only two markable vertices, $v_1$ and $v_6$. With Player 2 unable to make one of them unmarkable, Player 1 will win by coloring the other of these two vertices with a 3, similar to Case 2 for $P_5$. This strategy is demonstrated in Figure \ref{fig:P_6game1}.
\begin{figure}[H]
\centering
\begin{tikzpicture}[scale=0.7]
    \node[draw,circle] (v1) {$v_1$};
    
    \path (v1) ++(0:2) node [draw,circle] (v2) {$v_2$};
    \node at ([shift={(90:0.4)}]v2.90) {1}; 

    \path (v2) ++(0:2) node [draw,circle] (v3) {$v_3$};
    
    \path (v3) ++(0:2) node [draw,circle] (v4) {$v_4$};
    \node at ([shift={(90:0.4)}]v4.90) {1}; 
    
    \path (v4) ++(0:2) node [draw,circle] (v5) {$v_5$};
    \node at ([shift={(90:0.4)}]v5.90) {2};
    
    \path (v5) ++(0:2) node [draw,circle] (v6) {$v_6$};

    \draw (v1) -- (v2) -- (v3) -- (v4) -- (v5) -- (v6);
    \end{tikzpicture}
    
\caption{$P_6$, Case 2, where Player 1 is in a winning position.}
\label{fig:P_6game1}
\end{figure}
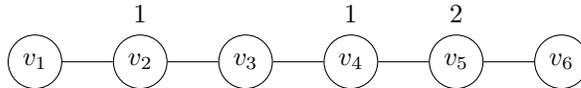
\end{proof}

At this point we digress slightly to illustrate that the requirement that the number of colors available be limited by the EDCN makes a difference in who has the winning strategy in some cases.

\begin{thm}\label{thm:numcolors}
    Player 2 has the winning strategy on $P_6$ if $\lambda(P_6)+1 = 4$ colors are allowed.
\end{thm}

\begin{proof}
    Player 2 has the winning strategy on $P_6$ by ensuring that after Player 2's first move, the game state is equivalent (up to symmetry and color permutation) to one of the two game states illustrated in Figure \ref{fig:P_6fourcolors}. To see this, consider two cases.

    \begin{figure}[H]
\centering
\resizebox{0.6\textwidth}{!}{%
\begin{circuitikz}
\tikzstyle{every node}=[font=\Huge]
\draw [, line width=0.8pt ] (0,12.5) circle (0.75cm) node {\Huge $v_1$} ;
\draw [, line width=0.8pt ] (5,12.5) circle (0.75cm) node {\Huge $v_2$} ;
\draw [, line width=0.8pt ] (10,12.5) circle (0.75cm) node {\Huge $v_3$} ;
\draw [, line width=0.8pt ] (15,12.5) circle (0.75cm) node {\Huge $v_4$} ;
\draw [, line width=0.8pt ] (20,12.5) circle (0.75cm) node {\Huge $v_5$} ;
\draw [, line width=0.8pt ] (25,12.5) circle (0.75cm) node {\Huge $v_6$} ;
\draw [line width=1.4pt, short] (0.75,12.5) -- (4.25,12.5);
\draw [line width=1.4pt, short] (5.75,12.5) -- (9.25,12.5);
\draw [line width=1.4pt, short] (10.75,12.5) -- (14.25,12.5);
\draw [line width=1.4pt, short] (15.75,12.5) -- (19.25,12.5);
\draw [line width=1.4pt, short] (20.75,12.5) -- (24.25,12.5);

\draw [, line width=0.8pt ] (0,8.75) circle (0.75cm) node {\Huge $v_1$} ;
\draw [, line width=0.8pt ] (5,8.75) circle (0.75cm) node {\Huge $v_2$} ;
\draw [, line width=0.8pt ] (10,8.75) circle (0.75cm) node {\Huge $v_3$} ;
\draw [, line width=0.8pt ] (15,8.75) circle (0.75cm) node {\Huge $v_4$} ;
\draw [, line width=0.8pt ] (20,8.75) circle (0.75cm) node {\Huge $v_5$} ;
\draw [, line width=0.8pt ] (25,8.75) circle (0.75cm) node {\Huge $v_6$} ;
\draw [line width=1.4pt, short] (0.75,8.75) -- (4.25,8.75);
\draw [line width=1.4pt, short] (5.75,8.75) -- (9.25,8.75);
\draw [line width=1.4pt, short] (10.75,8.75) -- (14.25,8.75);
\draw [line width=1.4pt, short] (15.75,8.75) -- (19.25,8.75);
\draw [line width=1.4pt, short] (20.75,8.75) -- (24.25,8.75);
\node [font=\Huge] at (0,13.75) {$1$};
\node [font=\Huge] at (5,13.75) {$1$};
\node [font=\Huge] at (0,10) {$1$};
\node [font=\Huge] at (15,10) {$1$};
\end{circuitikz}
}%
\vspace{0.3cm}

\caption{Game states after second move that will put Player 2 in a winning position on $P_6$ with $\lambda(P_6)+1=4$ colors.}
\label{fig:P_6fourcolors}
\end{figure}
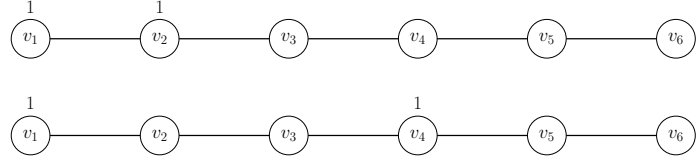

    \item\underline{Case 1:} \emph{Player 1's first move is on $v_1$, $v_2$, $v_5$ or $v_6$.} Player 2's first move should be to play the same color on $v_2$, $v_1$, $v_6$, or $v_5$, respectively, ensuring that the game state is equivalent to the top graph in Figure \ref{fig:P_6fourcolors}. Without loss of generality, assume vertices $v_1$ and $v_2$ are color 1 after Player 2's first move, as in the figure. These leads to four subcases based on where Player 1 plays next and what color is played.

       \begin{enumerate}
        \item[1.1.] \emph{Player 1 plays on $v_4$ or $v_6$ with color 1:} Player 2 should play on $v_6$ or $v_4$, respectively, using color 1 as well. This results in both vertices $v_3$ and $v_5$ becoming unmarkable, meaning Player 2 wins.

        \item[1.2.] \emph{Player 1 plays on $v_5$ with color 1:} All remaining vertices are adjacent to a 1, and a $\{1,1\}$ edge has already been created, so the color 1 can no longer be played. Further, all three remaining colors $2$, $3$, and $4$ must be played because playing one of these colors on more than one vertex will result in a duplicate edge color due to every remaining vertex being adjacent to a 1. By parity, this results in a win for Player 2.
        
        \item[1.3.] \emph{Player 1 plays on $v_5$ with a color other than 1:} Without loss of generality, suppose Player 1 plays color 2. Player 2 should play color 1 on $v_6$. Note that at this point there is both a $\{1,1\}$ edge and a $\{1,2\}$ edge. Since one of the remaining vertices is adjacent to a 1 and the other is adjacent to a 2, neither of these vertices may be colored 1 or 2 without creating a duplicate edge color. This implies that Player 1 must play one of the remaining two colors as their next move. Player 2 may play the last color as the final move, winning the game.

        \item[1.4.] \emph{Player 1 plays on a vertex other than $v_5$ with a color other than 1:} Without loss of generality, suppose Player 1 plays color 2. Player 2 should play color 1 on $v_5$. In each case, a $\{1,2\}$ edge has been created either by Player 1 (if they played on $v_3$), or by Player 2. Additionally, every remaining vertex is adjacent to a 1 and there is already a $\{1,1\}$ edge, so neither 1 nor 2 may be played on any remaining vertices. This means Player 1 must play one of the remaining two colors as their next move. Player 2 may play the last color as the final move, winning the game.
    \end{enumerate}

    \item\underline{Case 2:} \emph{Player 1's first move is on $v_3$ or $v_4$.} Player 2's first move should be to play the same color on $v_6$ or $v_1$, respectively, ensuring that the game state is equivalent to the bottom graph in Figure \ref{fig:P_6fourcolors}. Without loss of generality, assume vertices $v_1$ and $v_4$ are color 1 after Player 2's first move, as in the figure. These leads to five subcases based on where Player 1 plays next and what color they play.

    \begin{enumerate}
        \item[2.1.] \emph{Player 1 plays on a vertex other than $v_5$ with color 1:} Player 2 should play a 1 on $v_6$ if Player 1 did not, and play a 1 on $v_2$ if Player 1 did. As in Case 1.1, two unmarkable vertices have been created, ending the game. Player 2 is therefore the winner.
        
        \item[2.2.] \emph{Player 1 plays on $v_5$ with color 1:} All remaining vertices are adjacent to a 1. The argument follows as in Case 1.2.

        \item[2.3.] \emph{Player 1 plays on $v_2$ or $v_3$ with a color other than 1:}  Without loss of generality, suppose Player 1 plays color 2. Player 2 should play color 1 on $v_5$. At this point, a $\{1,2\}$ edge and a $\{1,1\}$ edge have been created. Additionally, one of the remaining vertices is adjacent to only a 1 while the other is adjacent to a 1 and a 2. This implies neither of the remaining vertices may be colored with a 1 or 2. Furthermore, this implies that Player 1 must play one of the remaining two colors as their next move. Player 2 may play the last color as the final move, winning the game.
        
        \item[2.4] \emph{Player 1 plays on $v_5$ with a color other than 1:}  Without loss of generality, suppose Player 1 plays color 2. Player 2's next move should be to play color 3 on $v_2$. At this point, the edges created are $\{1,3\}$ and $\{1,2\}$. Note then that a 3 cannot be played on $v_3$ because that would make a second $\{1,3\}$ edge. Additionally, the colors adjacent to the remaining vertices are all distinct: $v_3$ is adjacent to 1 and 3 while $v_6$ is adjacent to 2. This means that Player 2 can play a 4 on the remaining vertex regardless of Player's move, winning the game. 

        \item[2.5] \emph{Player 1 plays on $v_6$ with a color other than 1:}  Without loss of generality, suppose Player 1 plays color 2. As in Case 2.4, Player 2's next move should be to play color 3 on $v_2$. At this point the only edge created is a $\{1,3\}$ edge, but this implies that a 3 may not be played on $v_3$ nor $v_5$ without creating a duplicate $\{1,3\}$ edge. If Player 1 then plays a 4 on either vertex, Player 2 can play a 2 on the other. Otherwise, Player 2 can play the unused 4 on the remaining vertex to win the game since no vertex was made unmarkable.
    \end{enumerate}

    Therefore, because Player 2 can win in every possible case, Player 2 has the winning strategy on $P_6$ if $\lambda(P_6)+1 = 4$ colors are allowed.
\end{proof}

\subsection{Cycles}\label{subsec:cycles}

\noindent In this section we investigate winning strategies on cycles $C_n$ for $4\leq n\leq 7$. Because $C_3 = K_3$ we know already that this is a Player 2 winning graph via Theorem \ref{thm:complete}. In what follows, we show that $C_4$ and $C_7$ are Player 2 winning graphs as well, while $C_5$ and $C_6$ are Player 1 winning graphs.

Note that we assume the vertices $v_1,v_2,\dots v_n$ in each cycle $C_n$ are labeled in a clockwise manner.

\begin{thm}\label{thm:C4}
    Player 2 has the winning strategy on $C_4$.
\end{thm}
\begin{proof}
    Note that $\lambda(C_4) = 3$ from \cite{frank1982edcn}. Assume we have a cycle $C_4$ of four vertices labeled $v_1$ through $v_4$. Without loss of generality, assume that Player 1 plays the color $1$ at vertex $v_1$. Then, Player 2 should play the same color on $v_3$. This move makes the only other two vertices unmarkable 
    Therefore, since Player 2 was the last player to play a move, Player 2 wins. This is demonstrated in Figure \ref{fig:C_4game}.
    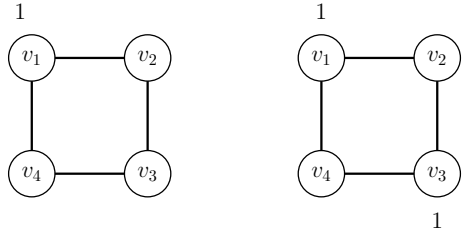
\begin{figure}[H]
    \centering
    \resizebox{0.4\textwidth}{!}{%
    \begin{circuitikz}
    \tikzstyle{every node}=[font=\LARGE]
    \draw [, line width=0.8pt ] (-3.75,16.25) circle (0.5cm) node {\LARGE $v_1$} ;
    \draw [, line width=0.8pt ] (-1.25,16.25) circle (0.5cm) node {\LARGE $v_2$} ;
    \draw [, line width=0.8pt ] (-1.25,13.75) circle (0.5cm) node {\LARGE $v_3$} ;
    \draw [, line width=0.8pt ] (-3.75,13.75) circle (0.5cm) node {\LARGE $v_4$} ;
    \draw [, line width=0.8pt ] (2.5,16.25) circle (0.5cm) node {\LARGE $v_1$} ;
    \draw [, line width=0.8pt ] (2.5,13.75) circle (0.5cm) node {\LARGE $v_4$} ;
    \draw [, line width=0.8pt ] (5,13.75) circle (0.5cm) node {\LARGE $v_3$} ;
    \draw [, line width=0.8pt ] (5,16.25) circle (0.5cm) node {\LARGE $v_2$} ;
    \draw [line width=1.4pt, short] (-3.75,15.75) -- (-3.75,14.25);
    \draw [line width=1.4pt, short] (-3.25,16.25) -- (-1.75,16.25);
    \draw [line width=1.4pt, short] (-1.25,15.75) -- (-1.25,14.25);
    \draw [line width=1.4pt, short] (-3.25,13.75) -- (-1.75,13.75);
    \draw [line width=1.4pt, short] (2.5,15.75) -- (2.5,14.25);
    \draw [line width=1.4pt, short] (3,16.25) -- (4.5,16.25);
    \draw [line width=1.4pt, short] (5,15.75) -- (5,14.25);
    \draw [line width=1.4pt, short] (3,13.75) -- (4.5,13.75);
    \node [font=\LARGE] at (-4,17.25) {$1$};
    \node [font=\LARGE] at (2.5,17.25) {$1$};
    \node [font=\LARGE] at (5,12.75) {$1$};
    \end{circuitikz}
    }%
    \caption{$C_4$, where Player 2 is in a  winning position. }\label{fig:C_4game}
    \end{figure}
\end{proof}

\begin{thm}\label{thm:C5}
    Player 1 has the winning strategy on $C_5$.
\end{thm}

\begin{proof}
     Players have the ability to play three different colors on this graph, as $\lambda(C_5) = 3$ \cite{frank1982edcn}. Assume, without loss of generality, that Player 1 plays the color $1$ at vertex $v_1$. There are two cases. 

    \item\underline{Case 1:} \emph{Player 2 plays another $1$.} No matter where Player 2 plays this $1$, Player 1 can play a terminating move by playing another $1$ in a remaining vertex to create a configuration isomorphic to the one displayed in Figure \ref{fig:C_5game1}. Doing so make the remaining two vertices unmarkable, so Player 1 wins.
    
    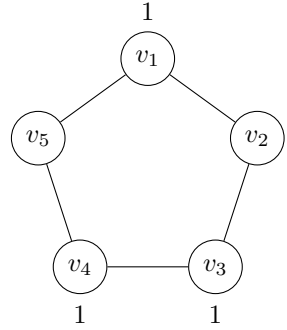
\begin{figure}[H]
    \centering
    
    \begin{tikzpicture}[scale=0.9]
    \node[draw,circle] (v1) {$v_1$};
    \node at ([shift={(90:0.3)}]v1.90) {1}; 
    
    \path (v1) ++(-36:2) node [draw,circle] (v2) {$v_2$};
    
    \path (v2) ++(-108:2) node [draw,circle] (v3) {$v_3$};
    \node at ([shift={(270:0.3)}]v3.270) {1}; 
    
    \path (v3) ++(180:2) node [draw,circle] (v4) {$v_4$};
    \node at ([shift={(270:0.3)}]v4.270) {1}; 

    \path (v4) ++(108:2) node [draw,circle] (v5) {$v_5$};
    
    \draw (v1) -- (v2) -- (v3) -- (v4) -- (v5) -- (v1);
    \end{tikzpicture}
    
    \vspace{0.5cm}
    \caption{$C_5$, Case 1, where Player 1 just won.}\label{fig:C_5game1}
    \end{figure}

    \item\underline{Case 2:} \emph{Player 2 plays a different color.} If Player 2 plays another color, say $2$, on any vertex, Player 1 should then create a path of three vertices of colors $1$, $1$, $2$ by playing another $1$. This can always be done since every vertex is a distance of $2$ or less away from $v_1$. By creating this path, there will be two vertices left, but a third color, $3$, can be played in both of them. Player 2 cannot play another $2$ to make the last vertex unmarkable because this move would create two $\{1, 2\}$ edges. Therefore, no matter what Player 2 plays for the second move, Player 1 can play a $3$ on the last remaining vertex. This completes the graph, and Player 1 has the last move and wins the game. This case is shown below in Figure \ref{fig:C_5game2}. 
    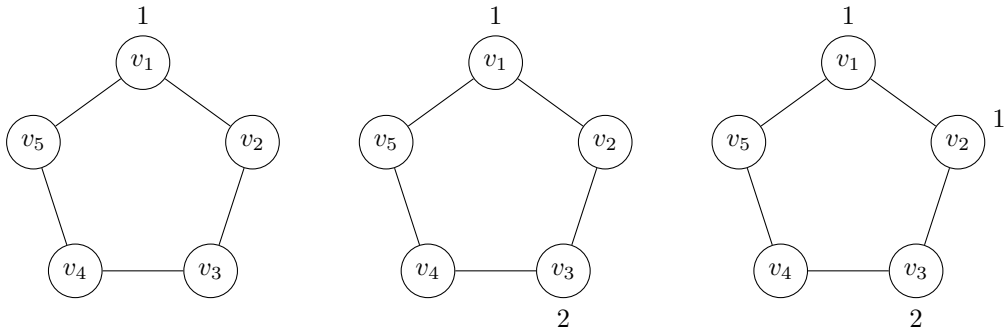
\begin{figure}[H]
    \centering
    
    \begin{minipage}{0.3\textwidth}
    \begin{tikzpicture}[scale=0.9]
    \node[draw,circle] (v1) {$v_1$};
    \node at ([shift={(90:0.3)}]v1.90) {1}; 
    
    \path (v1) ++(-36:2) node [draw,circle] (v2) {$v_2$};
    
    \path (v2) ++(-108:2) node [draw,circle] (v3) {$v_3$};
    \node at ([shift={(270:0.3)}]v3.270) {\phantom{2}};
    
    \path (v3) ++(180:2) node [draw,circle] (v4) {$v_4$};

    \path (v4) ++(108:2) node [draw,circle] (v5) {$v_5$};
    
    \draw (v1) -- (v2) -- (v3) -- (v4) -- (v5) -- (v1);
    \end{tikzpicture}
    \end{minipage}
    \begin{minipage}{0.3\textwidth}
    \begin{tikzpicture}[scale=0.9]
    \node[draw,circle] (v1) {$v_1$};
    \node at ([shift={(90:0.3)}]v1.90) {1}; 
    
    \path (v1) ++(-36:2) node [draw,circle] (v2) {$v_2$};
    
    \path (v2) ++(-108:2) node [draw,circle] (v3) {$v_3$};
    \node at ([shift={(270:0.3)}]v3.270) {2};
    
    \path (v3) ++(180:2) node [draw,circle] (v4) {$v_4$};

    \path (v4) ++(108:2) node [draw,circle] (v5) {$v_5$};
    
    \draw (v1) -- (v2) -- (v3) -- (v4) -- (v5) -- (v1);
    \end{tikzpicture}
    \end{minipage}
    \begin{minipage}{0.3\textwidth}
    \begin{tikzpicture}[scale=0.9]
    \node[draw,circle] (v1) {$v_1$};
    \node at ([shift={(90:0.3)}]v1.90) {1}; 
    
    \path (v1) ++(-36:2) node [draw,circle] (v2) {$v_2$};
    \node at ([shift={(30:0.3)}]v2.30) {1}; 
    
    \path (v2) ++(-108:2) node [draw,circle] (v3) {$v_3$};
    \node at ([shift={(270:0.3)}]v3.270) {2}; 
    
    \path (v3) ++(180:2) node [draw,circle] (v4) {$v_4$};

    \path (v4) ++(108:2) node [draw,circle] (v5) {$v_5$};
    
    \draw (v1) -- (v2) -- (v3) -- (v4) -- (v5) -- (v1);
    \end{tikzpicture}
    \end{minipage}
    
    \vspace{0.5cm}
    \caption{$C_5$, Case 2, where Player 1 is in a winning position .}\label{fig:C_5game2}
    \end{figure}

\bigskip
\end{proof}

\begin{thm}\label{thm:C6}
    Player 1 has the winning strategy on $C_6$.
\end{thm}

\begin{proof}
     Players have the ability to play three different colors on this graph, as $\lambda(C_6) = 3$ \cite{frank1982edcn}. Assume without loss of generality that Player 1 plays the color $1$ at vertex $v_1$. No matter what Player 2 plays, the strategy for Player 1 remains the same: play a move that creates one unmarkable vertex. However, there remains several cases based on where and what Player 2 plays.

    \item\underline{Case 1:} \emph{Player 2 plays another $1$.} 
    This leads to three subcases based on where Player 2 chooses to play the $1$.
    \begin{enumerate}
        \item[1.1.] \emph{Player 2 plays on $v_2$ or $v_6$:} In this case, Player 1's optimal move would be to play another $1$ on $v_4$. This move makes one unmarkable vertex, and leaves only two vertices markable. However, since no more $1$s can be played and neither $2$ nor $3$ have been played yet, whichever color Player 2 chooses to play, Player 1 will win by playing the other color on the last remaining vertex, therefore winning the game.
        \item[1.2.] \emph{Player 2 plays on $v_3$ or $v_5$:} Player 1 will win in this case by playing another $1$ on the other vertex. This means that if Player 2 played on $v_3$, Player 1 would play on $v_5$ and vice versa. This leaves the remaining vertices unmarkable. Therefore, since the last move was made by Player 1, this move wins the game.
        \item[1.3.] \emph{Player 2 plays on $v_4$:} This case is the reverse of when Player 2 played on $v_2$ or $v_6$. Therefore, Player 1 has the winning move in this situation by playing on either $v_2$ or $v_6$ with another $1$. This creates an identical game state to that described in 1.1, so Player 1 wins this game board. This case is shown in Figure \ref{fig:C_6game1}.
    \end{enumerate}
    
\begin{figure}[H]
\centering
\begin{minipage}{0.3\textwidth}
    \begin{tikzpicture}[scale=0.9]
    \node[draw,circle] (v1) {$v_1$};
    \node at ([shift={(90:0.3)}]v1.90) {1}; 
    
    \path (v1) ++(0:2) node [draw,circle] (v2) {$v_2$};
    
    \path (v2) ++(-60:2) node [draw,circle] (v3) {$v_3$};
    
    \path (v3) ++(-120:2) node [draw,circle] (v4) {$v_4$};
    \node at ([shift={(270:0.3)}]v4.270) {}; 

    \path (v4) ++(180:2) node [draw,circle] (v5) {$v_5$};
    
    \path (v5) ++(120:2) node [draw,circle] (v6) {$v_6$};
    
    \draw (v1) -- (v2) -- (v3) -- (v4) -- (v5) -- (v6) -- (v1);
    \end{tikzpicture}
    \end{minipage}\hspace{0.3cm}
    \begin{minipage}{0.3\textwidth}
    \begin{tikzpicture}[scale=0.9]
    \node[draw,circle] (v1) {$v_1$};
    \node at ([shift={(90:0.3)}]v1.90) {1}; 
    
    \path (v1) ++(0:2) node [draw,circle] (v2) {$v_2$};
    
    \path (v2) ++(-60:2) node [draw,circle] (v3) {$v_3$};
    
    \path (v3) ++(-120:2) node [draw,circle] (v4) {$v_4$};
    \node at ([shift={(270:0.3)}]v4.270) {1}; 

    \path (v4) ++(180:2) node [draw,circle] (v5) {$v_5$};
    
    \path (v5) ++(120:2) node [draw,circle] (v6) {$v_6$};
    
    \draw (v1) -- (v2) -- (v3) -- (v4) -- (v5) -- (v6) -- (v1);
    \end{tikzpicture}
    \end{minipage}\hspace{0.3cm}
    \begin{minipage}{0.3\textwidth}
    \begin{tikzpicture}[scale=0.9]
    \node[draw,circle] (v1) {$v_1$};
    \node at ([shift={(90:0.3)}]v1.90) {1}; 
    
    \path (v1) ++(0:2) node [draw,circle] (v2) {$v_2$};
    \node at ([shift={(90:0.3)}]v2.90) {1}; 
    
    \path (v2) ++(-60:2) node [draw,circle] (v3) {$v_3$};
    
    \path (v3) ++(-120:2) node [draw,circle] (v4) {$v_4$};
    \node at ([shift={(270:0.3)}]v4.270) {1}; 

    \path (v4) ++(180:2) node [draw,circle] (v5) {$v_5$};
    
    \path (v5) ++(120:2) node [draw,circle] (v6) {$v_6$};
    
    \draw (v1) -- (v2) -- (v3) -- (v4) -- (v5) -- (v6) -- (v1);
    \end{tikzpicture}
    \end{minipage}
    \vspace{0.5cm}
    
\caption{$C_6$, Case 1.3, where Player 1 is in a winning position.}
\label{fig:C_6game1}
\end{figure}
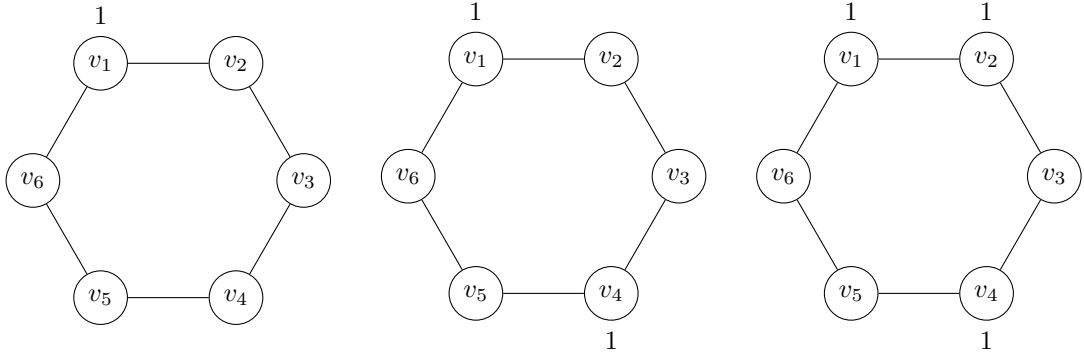

    \item\underline{Case 2:} \emph{Player 2 plays a different color.} Assume without loss of generality that Player 2 plays a $2$. This also leads to three subcases depending on where Player 2 plays the $2$.
    \begin{enumerate}
        \item[2.1.] \emph{Player 2 plays on $v_2$ or $v_6$:} Player 1's optimal move in this case is to play another $2$ at $v_4$. Doing so will create one unmarkable vertex and leave two markable vertices. Neither of these vertices can be made unmarkable because there is already a $\{1, 2\}$ edge. For these remaining vertices, $3$ can be played on both no matter what Player 2 chooses to play. Therefore, by then playing a $3$ as the final move, Player 1 wins.
        \item[2.2.] \emph{Player 2 Plays on $v_3$ or $v_5$:} Player 1's winning strategy in this case is to play a $1$ on the other vertex. This means that if Player 2 plays on $v_3$, Player 1 will play on $v_5$ and vice versa. This move creates one unmarkable vertex between itself and $v_1$. The only remaining markable vertices are adjacent to Player 2's $2$. For either of these vertices, if Player 2 plays a $1$ or a $2$, Player 1 wins by playing a $3$ on the last remaining vertex. Similarly, if Player 2 plays a $3$, Player 1 wins by playing a $1$ or a $2$ on the last vertex. This strategy is demonstrated in Figure \ref{fig:C_6game2}.
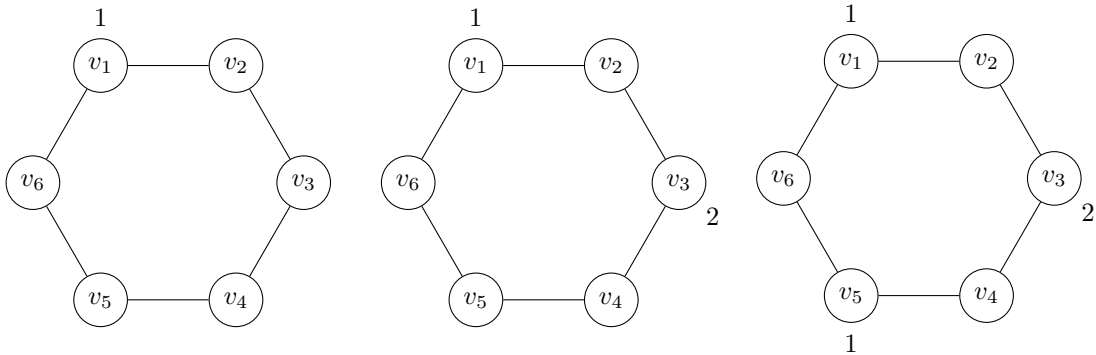
\begin{figure}[H]
\centering

\begin{minipage}{0.3\textwidth}
    \begin{tikzpicture}[scale=0.9]
    \node[draw,circle] (v1) {$v_1$};
    \node at ([shift={(90:0.3)}]v1.90) {1}; 
    
    \path (v1) ++(0:2) node [draw,circle] (v2) {$v_2$};
    
    \path (v2) ++(-60:2) node [draw,circle] (v3) {$v_3$};
    
    \path (v3) ++(-120:2) node [draw,circle] (v4) {$v_4$};
    \node at ([shift={(270:0.3)}]v4.270) {}; 

    \path (v4) ++(180:2) node [draw,circle] (v5) {$v_5$};
    
    \path (v5) ++(120:2) node [draw,circle] (v6) {$v_6$};
    
    \draw (v1) -- (v2) -- (v3) -- (v4) -- (v5) -- (v6) -- (v1);
    \end{tikzpicture}
    \end{minipage}\hspace{0.3cm}
    \begin{minipage}{0.3\textwidth}
    \begin{tikzpicture}[scale=0.9]
    \node[draw,circle] (v1) {$v_1$};
    \node at ([shift={(90:0.3)}]v1.90) {1}; 
    
    \path (v1) ++(0:2) node [draw,circle] (v2) {$v_2$};
    
    \path (v2) ++(-60:2) node [draw,circle] (v3) {$v_3$};
    \node at ([shift={(315:0.3)}]v3.315) {2};
    
    \path (v3) ++(-120:2) node [draw,circle] (v4) {$v_4$};
    \node at ([shift={(270:0.3)}]v4.270) {}; 

    \path (v4) ++(180:2) node [draw,circle] (v5) {$v_5$};
    
    \path (v5) ++(120:2) node [draw,circle] (v6) {$v_6$};
    
    \draw (v1) -- (v2) -- (v3) -- (v4) -- (v5) -- (v6) -- (v1);
    \end{tikzpicture}
    \end{minipage}\hspace{0.3cm}
    \begin{minipage}{0.3\textwidth}
    \begin{tikzpicture}[scale=0.9]
    \node[draw,circle] (v1) {$v_1$};
    \node at ([shift={(90:0.3)}]v1.90) {1}; 
    
    \path (v1) ++(0:2) node [draw,circle] (v2) {$v_2$};
    
    \path (v2) ++(-60:2) node [draw,circle] (v3) {$v_3$};
    \node at ([shift={(315:0.3)}]v3.315) {2};
    
    \path (v3) ++(-120:2) node [draw,circle] (v4) {$v_4$};

    \path (v4) ++(180:2) node [draw,circle] (v5) {$v_5$};
    \node at ([shift={(270:0.3)}]v5.270) {1};
    
    \path (v5) ++(120:2) node [draw,circle] (v6) {$v_6$};
    
    \draw (v1) -- (v2) -- (v3) -- (v4) -- (v5) -- (v6) -- (v1);
    \end{tikzpicture}
    \end{minipage}
    \vspace{0.5cm}
    
\caption{$C_6$, Case 2.2, where Player 1 is in a winning position.}
\label{fig:C_6game2}
\end{figure}
        \item[2.3.] \emph{Player 2 plays on $v_4$:} Player 1 has the winning strategy in this case by playing a $2$ on either $v_2$ or $v_6$. This creates an identical game state to that described in 2.1, so therefore Player 1 will win in this case.
    \end{enumerate}
    
\end{proof}

\begin{thm}\label{thm:C7}
    Player 2 has the winning strategy on $C_7$.
\end{thm}

\begin{proof}
    Players have the ability to play 4 different colors on this graph, as $\lambda(C_7) = 4$ \cite{frank1982edcn}. Assume, without loss of generality, that Player 1 plays $1$ on vertex $v_1$. Player 2 plays another $1$ on vertex $v_2$. Then, no matter where or what Player 1 plays, Player 2's winning move is to create an unmarkable vertex. 
    \item\underline{Case 1:} \emph{Player 1 plays another $1$.} By playing another $1$, Player 1 cannot play on either $v_3$ or $v_7$. This case has 2 subcases based on where this color is played.
    \begin{enumerate}
        \item[1.1.] \emph{Player 1 plays on $v_4$ or $v_6$.} Playing another $1$ in either of these locations creates one unmarkable vertex. So to win, Player 2 plays a $1$ on the other vertex. This means that if Player 1 chose to play on $v_4$, Player 2 would play on $v_6$ and vice versa. This move then makes all remaining vertices unmarkable, as every other vertex is adjacent to two $1$s. Thus, Player 2 wins. This stratgey is shown below in Figure \ref{fig:C_7game1}.
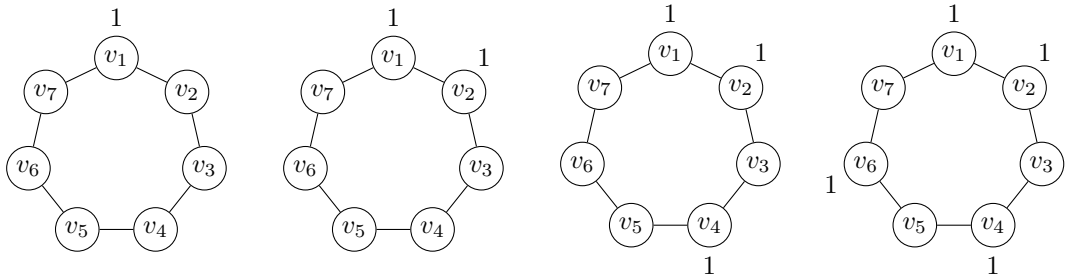
\begin{figure}[H]
    \begin{minipage}{0.2\textwidth}
    \begin{tikzpicture}[scale=1.2]
    \node[] (c) {};
    \path (c) ++(90:1) node [draw,circle,inner sep=0.7mm] (v1) {$v_1$};
    \path (c) ++(90+360/7:1) node [draw,circle,inner sep=0.7mm] (v7) {$v_7$};
    \path (c) ++(90+2*360/7:1) node [draw,circle,inner sep=0.7mm] (v6) {$v_6$};
    \path (c) ++(90+3*360/7:1) node [draw,circle,inner sep=0.7mm] (v5) {$v_5$};
    \path (c) ++(90+4*360/7:1) node [draw,circle,inner sep=0.7mm] (v4) {$v_4$};
    \path (c) ++(90+5*360/7:1) node [draw,circle,inner sep=0.7mm] (v3) {$v_3$};
    \path (c) ++(90+6*360/7:1) node [draw,circle,inner sep=0.7mm] (v2) {$v_2$};
    
    \draw (v1) -- (v2) -- (v3) -- (v4) -- (v5) -- (v6) -- (v7) -- (v1);

    \node at ([shift={(90:0.2)}]v1.90) {1}; 
    \node at ([shift={(60:0.2)}]v2.60) {}; 
    \node at ([shift={(270:0.2)}]v4.270) {};
    \node at ([shift={(210:0.2)}]v6.210) {}; 
    
    \end{tikzpicture}
    \end{minipage}
    \hspace{0.4cm}
    \begin{minipage}{0.2\textwidth}
    \begin{tikzpicture}[scale=1.2]
    \node[] (c) {};
    \path (c) ++(90:1) node [draw,circle,inner sep=0.7mm] (v1) {$v_1$};
    \path (c) ++(90+360/7:1) node [draw,circle,inner sep=0.7mm] (v7) {$v_7$};
    \path (c) ++(90+2*360/7:1) node [draw,circle,inner sep=0.7mm] (v6) {$v_6$};
    \path (c) ++(90+3*360/7:1) node [draw,circle,inner sep=0.7mm] (v5) {$v_5$};
    \path (c) ++(90+4*360/7:1) node [draw,circle,inner sep=0.7mm] (v4) {$v_4$};
    \path (c) ++(90+5*360/7:1) node [draw,circle,inner sep=0.7mm] (v3) {$v_3$};
    \path (c) ++(90+6*360/7:1) node [draw,circle,inner sep=0.7mm] (v2) {$v_2$};
    
    \draw (v1) -- (v2) -- (v3) -- (v4) -- (v5) -- (v6) -- (v7) -- (v1);

    \node at ([shift={(90:0.2)}]v1.90) {1}; 
    \node at ([shift={(60:0.2)}]v2.60) {1}; 
    \node at ([shift={(270:0.2)}]v4.270) {};
    \node at ([shift={(210:0.2)}]v6.210) {}; 
    
    \end{tikzpicture}
    \end{minipage}
    \hspace{0.4cm}
    \begin{minipage}{0.2\textwidth}
    \begin{tikzpicture}[scale=1.2]
    \node[] (c) {};
    \path (c) ++(90:1) node [draw,circle,inner sep=0.7mm] (v1) {$v_1$};
    \path (c) ++(90+360/7:1) node [draw,circle,inner sep=0.7mm] (v7) {$v_7$};
    \path (c) ++(90+2*360/7:1) node [draw,circle,inner sep=0.7mm] (v6) {$v_6$};
    \path (c) ++(90+3*360/7:1) node [draw,circle,inner sep=0.7mm] (v5) {$v_5$};
    \path (c) ++(90+4*360/7:1) node [draw,circle,inner sep=0.7mm] (v4) {$v_4$};
    \path (c) ++(90+5*360/7:1) node [draw,circle,inner sep=0.7mm] (v3) {$v_3$};
    \path (c) ++(90+6*360/7:1) node [draw,circle,inner sep=0.7mm] (v2) {$v_2$};
    
    \draw (v1) -- (v2) -- (v3) -- (v4) -- (v5) -- (v6) -- (v7) -- (v1);

    \node at ([shift={(90:0.2)}]v1.90) {1}; 
    \node at ([shift={(60:0.2)}]v2.60) {1}; 
    \node at ([shift={(270:0.2)}]v4.270) {1};
    \node at ([shift={(210:0.2)}]v6.210) {}; 
    
    \end{tikzpicture}
    \end{minipage}
    \hspace{0.4cm}
    \begin{minipage}{0.2\textwidth}
    \begin{tikzpicture}[scale=1.2]
    \node[] (c) {};
    \path (c) ++(90:1) node [draw,circle,inner sep=0.7mm] (v1) {$v_1$};
    \path (c) ++(90+360/7:1) node [draw,circle,inner sep=0.7mm] (v7) {$v_7$};
    \path (c) ++(90+2*360/7:1) node [draw,circle,inner sep=0.7mm] (v6) {$v_6$};
    \path (c) ++(90+3*360/7:1) node [draw,circle,inner sep=0.7mm] (v5) {$v_5$};
    \path (c) ++(90+4*360/7:1) node [draw,circle,inner sep=0.7mm] (v4) {$v_4$};
    \path (c) ++(90+5*360/7:1) node [draw,circle,inner sep=0.7mm] (v3) {$v_3$};
    \path (c) ++(90+6*360/7:1) node [draw,circle,inner sep=0.7mm] (v2) {$v_2$};
    
    \draw (v1) -- (v2) -- (v3) -- (v4) -- (v5) -- (v6) -- (v7) -- (v1);

    \node at ([shift={(90:0.2)}]v1.90) {1}; 
    \node at ([shift={(60:0.2)}]v2.60) {1}; 
    \node at ([shift={(270:0.2)}]v4.270) {1};
    \node at ([shift={(210:0.2)}]v6.210) {1}; 
    
    \end{tikzpicture}
    \end{minipage}
    
\caption{$C_7$, Case 1.1, where Player 2 is in a winning position.}
\label{fig:C_7game1}
\end{figure}
        \item[1.2.] \emph{Player 1 plays on $v_5$.} By playing here, Player 1 creates a game state where every remaining unmarked vertex is adjacent to a $1$. Because there already exists a $\{1, 1\}$ edge, $1$ can no longer be played. Because every remaining vertex is adjacent to a $1$, each remaining color can only be played once. However, since the EDCN is 4, only three colors remain: $2$, $3$, and $4$. Thus, since Player 2 is the next to play, Player 2 will also be last one to play. At the end of the game there will be one remaining vertex, but no available colors that will not create a duplicate edge, so it is unmarkable. Therefore, by being the last to play, Player 2 wins.
    \end{enumerate}
    \item \underline{Case 2:} \emph{Player 1 plays another color.} Assume, without loss of generality, that this color is $2$. From here, there are 3 subcases. 
    \begin{enumerate}
        \item[2.1.] \emph{Player 1 plays on $v_3$ or $v_7$.} Player 2 should play another $2$ on $v_5$, resulting in a graph isomorphic to the one in Figure \ref{fig:C_7game2.1}, so we can assume this graph to finish the game. Player 2 makes $v_4$ an unmarkable vertex and leaves two remaining markable vertices, neither of which can be made unmarkable since there is already a $\{1, 2\}$ edge, so 1 cannot be played on $v_6$ and 2 cannot be played on $v_7$. Because only two colors have been used up to this point and the EDCN is 4, both $3$ and $4$ can be played on these remaining vertices. Therefore, whichever color Player 1 plays, Player 2 can simply play a $3$ on the other vertex, winning the game.
    \begin{figure}[H]
    \begin{tikzpicture}[scale=1.2]
    \node[] (c) {};
    \path (c) ++(90:1) node [draw,circle,inner sep=0.7mm] (v1) {$v_1$};
    \path (c) ++(90+360/7:1) node [draw,circle,inner sep=0.7mm] (v7) {$v_7$};
    \path (c) ++(90+2*360/7:1) node [draw,circle,inner sep=0.7mm] (v6) {$v_6$};
    \path (c) ++(90+3*360/7:1) node [draw,circle,inner sep=0.7mm] (v5) {$v_5$};
    \path (c) ++(90+4*360/7:1) node [draw,circle,inner sep=0.7mm] (v4) {$v_4$};
    \path (c) ++(90+5*360/7:1) node [draw,circle,inner sep=0.7mm] (v3) {$v_3$};
    \path (c) ++(90+6*360/7:1) node [draw,circle,inner sep=0.7mm] (v2) {$v_2$};
    
    \draw (v1) -- (v2) -- (v3) -- (v4) -- (v5) -- (v6) -- (v7) -- (v1);

    \node at ([shift={(90:0.2)}]v1.90) {1}; 
    \node at ([shift={(60:0.2)}]v2.60) {1}; 
    \node at ([shift={(330:0.2)}]v3.330) {2};
    \node at ([shift={(210:0.2)}]v5.210) {2}; 
    
    \end{tikzpicture}
    \caption{$C_7$, Case 2.2, where Player 2 is in a winning position.}\label{fig:C_7game2.1}
    \end{figure}
        
        \item[2.2.] \emph{Player 1 plays on $v_4$ or $v_6$.} Player 2's winning move in this case is to play another $2$ on the other vertex of the two, which results in the game board shown in Figure \ref{fig:C_7game}. This creates one unmarkable vertex at $v_5$, leaving the only markable vertices as $v_3$ and $v_7$. 
        If Player 1 chooses to play $2$, Player 2 to plays $3$ on the other vertex. If Player 1 does not play $2$, Player 2 plays $2$ on the other vertex. This completes to board, and Player 2 wins.
    \begin{figure}[H]
        
    \begin{tikzpicture}[scale=1.2]
    \node[] (c) {};
    \path (c) ++(90:1) node [draw,circle,inner sep=0.7mm] (v1) {$v_1$};
    \path (c) ++(90+360/7:1) node [draw,circle,inner sep=0.7mm] (v7) {$v_7$};
    \path (c) ++(90+2*360/7:1) node [draw,circle,inner sep=0.7mm] (v6) {$v_6$};
    \path (c) ++(90+3*360/7:1) node [draw,circle,inner sep=0.7mm] (v5) {$v_5$};
    \path (c) ++(90+4*360/7:1) node [draw,circle,inner sep=0.7mm] (v4) {$v_4$};
    \path (c) ++(90+5*360/7:1) node [draw,circle,inner sep=0.7mm] (v3) {$v_3$};
    \path (c) ++(90+6*360/7:1) node [draw,circle,inner sep=0.7mm] (v2) {$v_2$};
    
    \draw (v1) -- (v2) -- (v3) -- (v4) -- (v5) -- (v6) -- (v7) -- (v1);

    \node at ([shift={(90:0.2)}]v1.90) {1}; 
    \node at ([shift={(60:0.2)}]v2.60) {1}; 
    \node at ([shift={(270:0.2)}]v4.270) {2};
    \node at ([shift={(210:0.2)}]v6.210) {2}; 
    
    \end{tikzpicture}
    
        \caption{$C_7$, Case 2.2, where Player 2 is in a winning position.}\label{fig:C_7game}
        \end{figure}
        \item[2.3.] \emph{Player 1 plays on $v_5$.} Player 2's winning move in this case is to play another $2$ on either $v_7$ or $v_3$. This creates one unmarkable vertex between it and $v_5$. This game state is isomorphic to the one described in 2.1, so Player 2 will win in this case.
    \end{enumerate}
    
\end{proof}

\subsection{Triangular Chorded Cycles}\label{subsec:triangular-chorded}

\noindent A triangular chorded cycle $C^{\{1,3\}}_n$ is a cycle $C_n$ with a single chord added between vertices $v_1$ and $v_3$, forming a triangle within the cycle. These belong to the larger family of chorded cycles $C^{\{1,j\}}_n$ where a chord may be added between any two non-adjacent vertices.

We used the code mentioned in Section \ref{sec:computation} to determine the winning players for the triangular chorded cycles $C^{\{1,3\}}_n$ with $4 \leq n \leq 8$. The results for these graphs are included in Table \ref{table:tcc}. The EDCNs are results from \cite{fickes2021edge}.

\begin{table}[h!]
\centering
    \begin{tabular}{ |c|c|c|  }
        \hline
        Graph& EDCN& Winning Player\\
        \hline
        $C^{\{1,3\}}_4$ &4&Player 2\\
        $C^{\{1,3\}}_5$ &4&Player 1\\
        $C^{\{1,3\}}_6$ &4&Player 1\\
        $C^{\{1,3\}}_7$ &4&Player 2\\
        $C^{\{1,3\}}_8$ &5&Player 1\\
        \hline
    \end{tabular}
    \caption{Results for triangular chorded cycles}
    \label{table:tcc}
\end{table}

\subsection{Miscellaneous Graphs}\label{subsec:misc}
Alongside those that were previously discussed, we determined which player has the winning strategy for a variety of other graphs. These graphs are shown in Table \ref{table:strats}. The majority of the EDCNs of these graphs were checked by the code in \cite{edge_python}, and the winning strategies were checked by the code in \cite{edge_game_python}. It should be noted that $T_n$ represents the triangular ladder graph on $n$ vertices, that is, the graph whose adjacency matrix is given by the $n\times n$ Toeplitz matrix
\[
A_{ij} = 
\begin{cases}
    1 & \text{ if } |i-j|\in\{1,2\}\\
    0 & \text{ else}
\end{cases}.
\]

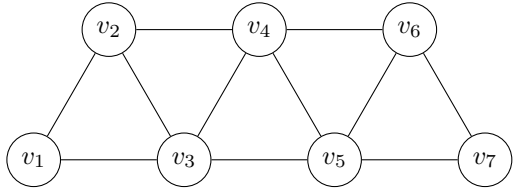
\begin{figure}[H]
    \centering
    \begin{tikzpicture}[scale=1]
    \node[draw,circle] (v1) {$v_1$};
    \path (v1) ++(60:2) node [draw,circle] (v2) {$v_2$};
    \path (v2) ++(-60:2) node [draw,circle] (v3) {$v_3$};
    \path (v3) ++(60:2) node [draw,circle] (v4) {$v_4$};
    \path (v4) ++(-60:2) node [draw,circle] (v5) {$v_5$};
    \path (v5) ++(60:2) node [draw,circle] (v6) {$v_6$};
    \path (v6) ++(-60:2) node [draw,circle] (v7) {$v_7$};
    
    \draw (v1) -- (v2) -- (v3) -- (v4) -- (v5) -- (v6) -- (v7) -- (v5) -- (v3) -- (v1);
    \draw (v2) -- (v4) -- (v6);
    \end{tikzpicture}
\caption{The triangular ladder graph on 7 vertices, $T_7$.}
\label{fig:T_7}
\end{figure}

\begin{table}[h!]
\centering
    \begin{tabular}{ |c|c|  }
        \hline
        Graph& Winning Player\\
        \hline
        $C_8$ &Player 2\\
        $P_7$ &Player 2\\
        Cube &Player 2\\
        Octahedron &Player 2\\
        Petersen Graph &Player 1\\
        Moser Spindle &Player 1\\
        Envelope Graph &Player 1\\
        $T_4$ &Player 2\\
        $T_5$ &Player 2\\
        $T_6$ &Player 2\\
        $T_7$ &Player 2\\
        $T_8$ &Player 2\\
        \hline
    \end{tabular}
    \vspace{0.2cm}
    \caption{Table of Winning Strategies for Various Graphs}
    \label{table:strats}
\end{table}

For most of the graphs in Table \ref{table:strats}, we were not able to identify a clear, easy-to-describe strategy. Rather, besides the first and sometimes second move, the optimal move for the winning player appears to be highly dependent on the previous player's move.
One example where this is not the case is the Moser Spindle graph. The Moser Spindle graph has a notable strategy where the player with the winning strategy, Player 1, wins the game by winning on subgraphs of the Moser Spindle. This is demonstrated below in the proof of Theorem \ref{MSTP}.

\begin{thm}\label{MSTP}
    For the Moser Spindle $M$, Player 1 has the winning strategy.
\end{thm}

\begin{proof}
    Let $M$ be the Moser Spindle graph, where the vertices are labeled as shown in Figure \ref{MSgame}. Players have the ability to play 6 colors on this graph as $\lambda(M) = 6$. Player 1's first move should be to play $1$ on $v_1$. For the rest of the strategy, imagine a vertical line through $v_1$ dividing $v_4$, $v_5$ and $v_6$ into one group and $v_2$, $v_3$, and $v_7$ into the other. Player 1's strategy is to eliminate the other group that Player 2 does not play on by playing either on $v_3$ or $v_4$. Without loss of generality, assume Player 2 plays on $v_2$ or $v_7$, since all game boards are isomorphic to one of those two moves. To follow the strategy, Player 1 would close off the other group by playing $1$ on $v_4$, making $v_5$ and $v_6$ unmarkable. If Player 2 played a 1 on $v_3$, the game is over and Player 1 has won. Otherwise, Player 2 has played color $i$ on $v_2$ and $v_3$ and $v_7$ are left as the only markable vertices. Note that every vertex is now adjacent to a vertex colored 1, so each color can be used at most once other than $i$, since a $\{1,i\}$ edge already exists. Neither of the remaining vertices can be eliminated, since a 1 cannot be played on $v_3$ as this would also create a duplicate $\{1,i\}$ edge. Hence, the board is in a favorable position for Player 1, since Player 2's turn begins with an even number of vertices left. Therefore, Player 1 will always win using this strategy. This strategy is demonstrated in Figure \ref{MSgame}.
\end{proof}
\begin{figure}[H]
\centering
\resizebox{1\textwidth}{!}{%
\begin{circuitikz}
\tikzstyle{every node}=[font=\Huge]
\draw [ line width=0.8pt ] (3.75,62.5) circle (0.75cm) node {\Huge $v_1$} ;
\draw [, line width=0.8pt ] (-1.25,60) circle (0.75cm) node {\Huge $v_5$} ;
\draw [, line width=0.8pt ] (8.75,60) circle (0.75cm) node {\Huge $v_2$} ;
\draw [, line width=0.8pt ] (5,57.5) circle (0.75cm) node {\Huge $v_7$} ;
\draw [, line width=0.8pt ] (2.5,57.5) circle (0.75cm) node {\Huge $v_6$} ;
\draw [, line width=0.8pt ] (-2.5,55) circle (0.75cm) node {\Huge $v_4$} ;
\draw [, line width=0.8pt ] (10,55) circle (0.75cm) node {\Huge $v_3$} ;
\draw [line width=1.4pt, short] (-1,60.75) -- (3,62.25);
\draw [line width=1.4pt, short] (8.5,60.75) -- (4.5,62.25);
\draw [line width=1.4pt, short] (-0.75,59.5) -- (1.75,57.75);
\draw [line width=1.4pt, short] (8.25,59.5) -- (5.75,57.75);
\draw [line width=1.4pt, short] (3.5,61.75) -- (2.75,58.25);
\draw [line width=1.4pt, short] (4,61.75) -- (4.75,58.25);
\draw [line width=1.4pt, short] (-1.75,59.5) -- (-2.5,55.75);
\draw [line width=1.4pt, short] (2,57) -- (-1.75,55.25);
\draw [line width=1.4pt, short] (5.5,57) -- (9.25,55);
\draw [line width=1.4pt, short] (9.25,59.5) -- (10,55.75);
\draw [line width=1.4pt, short] (-1.75,54.75) -- (9.25,54.75);

\draw [, line width=0.8pt ] (13.75,55) circle (0.75cm) node {\Huge $v_4$} ;
\draw [, line width=0.8pt ] (15,60) circle (0.75cm) node {\Huge $v_5$} ;
\draw [, line width=0.8pt ] (18.75,57.5) circle (0.75cm) node {\Huge $v_6$} ;
\draw [, line width=0.8pt ] (21.25,57.5) circle (0.75cm) node {\Huge $v_7$} ;
\draw [, line width=0.8pt ] (26.25,55) circle (0.75cm) node {\Huge $v_3$} ;
\draw [, line width=0.8pt ] (25,60) circle (0.75cm) node {\Huge $v_2$} ;
\draw [, line width=0.8pt ] (20,62.5) circle (0.75cm) node {\Huge $v_1$} ;
\draw [, line width=0.8pt ] (30,55) circle (0.75cm) node {\Huge $v_4$} ;
\draw [, line width=0.8pt ] (31.25,60) circle (0.75cm) node {\Huge $v_5$} ;
\draw [, line width=0.8pt ] (35,57.5) circle (0.75cm) node {\Huge $v_6$} ;
\draw [, line width=0.8pt ] (37.5,57.5) circle (0.75cm) node {\Huge $v_7$} ;
\draw [, line width=0.8pt ] (42.5,55) circle (0.75cm) node {\Huge $v_3$} ;
\draw [, line width=0.8pt ] (41,60) circle (0.75cm) node {\Huge $v_2$} ;
\draw [, line width=0.8pt ] (36.25,62.5) circle (0.75cm) node {\Huge $v_1$} ;
\draw [line width=1.4pt, short] (19.25,62.25) -- (15.25,60.75);
\draw [line width=1.4pt, short] (20.75,62.25) -- (25,60.75);
\draw [line width=1.4pt, short] (19.75,61.75) -- (19.25,58);
\draw [line width=1.4pt, short] (20.5,62) -- (21,58.25);
\draw [line width=1.4pt, short] (14.5,59.5) -- (13.75,55.75);
\draw [line width=1.4pt, short] (18.25,57) -- (14.5,55);
\draw [line width=1.4pt, short] (21.75,57) -- (25.5,55);
\draw [line width=1.4pt, short] (25.5,59.5) -- (26.25,55.75);
\draw [line width=1.4pt, short] (14.5,54.75) -- (25.5,54.75);
\draw [line width=1.4pt, short] (35.5,62.25) -- (31.5,60.75);
\draw [line width=1.4pt, short] (37,62.25) -- (40.5,60.5);
\draw [line width=1.4pt, short] (36,61.75) -- (35.25,58.25);
\draw [line width=1.4pt, short] (36.5,61.75) -- (37,58);
\draw [line width=1.4pt, short] (32,59.75) -- (34.5,58);
\draw [line width=1.4pt, short] (40.5,59.5) -- (38.25,57.75);
\draw [line width=1.4pt, short] (30.75,59.5) -- (30,55.75);
\draw [line width=1.4pt, short] (34.5,57) -- (30.75,55);
\draw [line width=1.4pt, short] (38,57) -- (41.75,55.25);
\draw [line width=1.4pt, short] (41.5,59.25) -- (42.5,55.75);
\draw [line width=1.4pt, short] (30.75,54.75) -- (41.75,54.75);
\draw [line width=1.4pt, short] (24.5,59.5) -- (22,57.75);
\draw [line width=1.4pt, short] (15.5,59.5) -- (18,57.75);
\node [font=\Huge] at (3.75,63.75) {1};
\node [font=\Huge] at (20,64) {1};
\node [font=\Huge] at (36,64) {1};
\node [font=\Huge] at (26,61) {$i$};
\node [font=\Huge] at (30,53.5) {1};
\node [font=\Huge] at (42,61) {$i$};
\end{circuitikz}
}%
\caption{Moser Spindle, where Player 1 is in a winning position.}
\label{MSgame}
\end{figure}
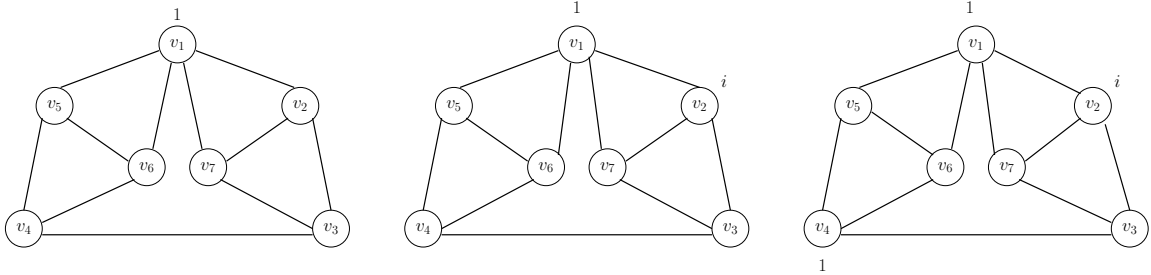

\section{Future Work}\label{sec:future-work}

Here we present a number of directions in which the game EDGe could be explored further. Since the EDCN for paths, cycles, and triangular chorded cycles are known \cite{frank1982edcn, al1988edge}, and since we were only able to determine which player has the winning strategy for particular cases of paths, cycles, and triangular chorded cycles, the following is an obvious direction this work could be extended.
\begin{question}
    Is there an overarching pattern for which player has the winning strategy for the families of paths, cycles, or triangular chorded cycles?
\end{question}
Formulas for the EDCN of spider graphs with three legs, generic chorded cycles, and petal graphs have also been found \cite{fickes2021edge}, so these classes of graphs could be explored as well.

\begin{question}
    Is there a pattern for which player has the winning strategy for the families of spider graphs, generic chorded cycles, or petal graphs?
\end{question}

In Table \ref{table:strats} we can see that Player 2 has the winning strategy on the triangular ladder graphs $T_n$ for $n\in \{4,5,6,7,8\}$. This observation leads to the following question.

\begin{question}
    Does Player 2 always have the winning strategy on $T_n$ for $n\geq 4$?
\end{question}

We are also interested in identifying overarching strategies for certain graph structures. In some cases, there may be a more generic strategy for a player to win, similar to the mirror-reverse strategy for symmetric graphs in the Game of Cycles \cite{alvarado2021game}. The example of the winning strategy in the Moser-Spindle is suggestive of a potential generic strategy to look for.

\begin{question}
    Are there graphs for which a player can win by winning on subgraphs, similar to the case of the Moser-Spindle?
\end{question}
In a similar vein, it seems possible that the winning strategy of a graph constructed via a binary graph operation could be dependent on the winning strategy of the initial graphs, in some cases. So we pose the following question as well.
\begin{question}
    Given two graphs $G_1$ and $G_2$ and a binary graph operation $*$, are there instances when the winning strategy on $G_1*G_2$ is dependent on the winning strategies on $G_1$ and $G_2$?
\end{question}
As all the graphs we looked at were connected, one binary operation that could be interesting to consider is unions of graphs, as this could result in disconnected graphs.

All graphs considered in this paper were also simple. However, it is possible to play EDGe on some types of non-simple graphs, such as those with loops. Playing on a graph multiple edges is likely less interesting because if two vertices are connected by multiple edges, at most one of those vertices could ever be legally colored. 
As a straightforward example a winning strategy for a non-simple graph with loops, consider the following.

\begin{thm}
    For complete graphs with loops on every vertex, $K^*_n$, with $n \ge 3$, Player 2 has the winning strategy if $n$ is even, and Player 1 has the winning strategy if $n$ is odd.
\end{thm}

\begin{proof}
    In $K^*_n$, since every vertex has a common neighbor with every other vertex, $\lambda(K^*_n) = n$. Once a color, say $1$, is played on any vertex, since there is a loop on every vertex, an edge of color $\{1, 1\}$ is created. Thus, the next move and all subsequent moves must be made with a distinct color. Therefore, the player that wins the game will be to be the last person to label a vertex. Consequently, if $n$ is odd Player 1 wins, and if $n$ is even, Player 2 wins.
\end{proof}

Notice that adding loops to a complete graph results in a different pattern of which player has the winning strategy (compare Theorem \ref{thm:complete}). Thus, one question that could be of interest is the following.
\begin{question}
    When and how does adding one or more loops to a simple graph affect which player has the winning strategy compared to the original simple graph?
\end{question}

Alternatively, in light of Theorem \ref{thm:numcolors}, one could relax the constraint that the number of colors available to the players is given by the EDCN.
\begin{question}
    In what cases does changing the number of available colors change who has the winning strategy for a graph?
\end{question}

Finally, it would be useful to identify ways to optimize the code mentioned in Section \ref{sec:computation} as the current code is infeasible to run on larger graphs. 
Two ways that the code could potentially be optimized include taking advantage of the fact that some colorings determine equivalent game states up to permutation of the colors, and taking advantage of the fact that some game states are equivalent up to graph isomorphisms. Additionally, one may be able to leverage the fact that after creating one or more unmarkable vertices, the game state may be equivalent to a game state on another game board. For example, $C_7$ with vertices $v_1$ and $v_3$ colored the same creates an unmarkable vertex and results in a game state equivalent to $P_6$ with $v_1$ and $v_6$ colored the same.
Such improvements could decrease the number of boards needed to construct the digraph, which could allow for the program to determine winning strategies for larger graphs in a more reasonable amount of time. Other improvements would be of interest as well.

\bibliographystyle{amsalpha}
\bibliography{references.bib}

\end{document}